\documentclass[12pt]{amsart}
\usepackage{amssymb, graphicx}
\usepackage[english]{babel}
\usepackage{indentfirst}
\usepackage[T1]{fontenc}
\usepackage[utf8]{inputenc}
\usepackage{lmodern}
\usepackage{tikz,tkz-base}
\usepackage{tikz-cd}
\usepackage{anysize}
\marginsize{2.5cm}{2.5cm}{4cm}{4cm}
\usetikzlibrary{matrix,arrows,decorations.pathmorphing}
\tikzset{commutative diagrams/column sep/my/.initial=0.01ex}

\newtheorem{theorem}{Theorem}[section]
\newtheorem{corollary}[theorem]{Corollary}

\theoremstyle{definition}
\newtheorem{definition}[theorem]{Definition}

\theoremstyle{remark}
\newtheorem{remark}[theorem]{\bf Remark}

\numberwithin{equation}{section}

\newcommand{\N}{\mathbb N}

\DeclareMathOperator{\Z}{\mathbb{Z}}
\DeclareMathOperator{\Q}{\mathbb{Q}}
\DeclareMathOperator{\F}{\mathbb{F}}

\DeclareMathOperator{\spec}{spec}

\makeatletter
\@namedef{subjclassname@2020}{%
  \textup{2020} Mathematics Subject Classification}
\makeatother
\tikzset{commutative diagrams/row sep/my/.initial=0.1ex}
\begin{document}
\title[Families of wild 1-motives]{Families of wild 1-motives}

%
\author[Grzegorz Banaszak]{Grzegorz Banaszak}
\address{Department of Mathematics and Computer Science, Adam Mickiewicz University, Uniwersytetu Poznańskiego 4, 
Pozna\'{n} 61614, Poland}
\email{banaszak@amu.edu.pl}
\urladdr{https://banaszak.faculty.wmi.amu.edu.pl/}

\author[Dorota Blinkiewicz]{Dorota Blinkiewicz}
\address{Department of Mathematics and Computer Science, Adam Mickiewicz University, Uniwersytetu Poznańskiego 4, 
Pozna\'{n} 61614, Poland}
\email{dorota.blinkiewicz@amu.edu.pl}
\urladdr{https://db.faculty.wmi.amu.edu.pl/}



\begin{abstract}
In this paper we present families of wild 1-motives, i.e., families of pairwise non-isomorphic Deligne 1-motives, over rings of $S$-integers $\mathcal{O}_{F,S}$, which have the same reductions to torsion 1-motives for all $v\notin S$. Our proof is based on a technical result concerning a local to global principle for multiple base discrete logarithm problem for arbitrary big bases.
\end{abstract}
\maketitle
\section{Introduction}
Let $F$ be a number field and $\mathcal{O}_F$ the ring of integers in $F.$ Let $S$ be a finite set of prime ideals of $\mathcal{O}_F$ and let $\mathcal{O}_{F,S}$ be the ring of $S$-integers. 

In this paper, in section \ref{sec:section 2}, we recall definitions of Deligne 1-motive and torsion 1-motive. Previously families of wild 1-motives were investigated for small ranks. 
In present paper, in section \ref{new constructions}, we consider families of wild 1-motives with arbitrary ranks. Main results of these constructions are presented in Theorems \ref{wild 1-motives for tori} and \ref{const2}. Proofs of these Theorems are based on the technical result -- Theorem \ref{thm}. In the proof of Theorem \ref{thm}, we also investigate the local to global multiple base discrete logarithm problem, which will be discussed in section \ref{appendix}.

Because of number of technical computations we will keep working in this paper with $F=\Q$ and $\mathcal{O}_F=\Z$ for the simplicity of presentation. Our computations can easily be extended for 1-motives over $\mathcal{O}_{F,S},$ for any number field $F,$ with $S$ such that $\mathcal{O}_{F,S}$ is a PID. 
\section{1-motives and torsion 1-motives}\label{sec:section 2}
Following \cite{J2} (cf. \cite{B-VRS}), let us briefly recall definitions of torsion 1-motive and 1-motive in the sense of P. Deligne.
\begin{definition}
Let $\mathcal{S}$ be a noetherian regular scheme and let $\mathcal{C}$ be a category of sheaves of commutative groups on the $fppf$ site over $\mathcal{S}$. 

\textbf{A torsion 1-motive over $\mathcal{S}$} is a diagram:
\[
\begin{tikzcd}
&&\mathcal{Y}\arrow{d}{u}&&\\
0\arrow{r}&\mathcal{X}\arrow{r}{i}&\mathcal{Z}\arrow{r}{\pi}&\mathcal{A}\arrow{r}&0,
\end{tikzcd}
\]
in the category $\mathcal{C}$, where:
\begin{itemize}
\item the sheaf $\mathcal{Y}$ fits into an exact sequence $0\rightarrow \mathcal{F}\rightarrow \mathcal{Y}\rightarrow \Lambda\rightarrow 0$, where $\mathcal{F}$ is a finite flat group scheme, and $\Lambda$ is a lattice.
\item the sheaf $\mathcal{X}$ fits into an exact sequence $0\rightarrow \mathcal{T}\rightarrow \mathcal{X} \rightarrow \mathcal{F}\rightarrow 0$, where $\mathcal{F}$ is a finite flat group scheme and $\mathcal{T}$ is a torus.
\item the sheaf $\mathcal{A}$ is an abelian scheme.
\item the sequence $0\rightarrow \mathcal{X} \rightarrow \mathcal{Z}\rightarrow \mathcal{A}\rightarrow 0$ is exact.
\end{itemize}
\end{definition}
\begin{remark}
We will use the following notation for a torsion 1-motive:
\[
[\mathcal{Y}\rightarrow \mathcal{Z}].
\]
\end{remark}
\begin{remark}
\textbf{A 1-motive in the sense of P.~Deligne} is defined as above, where $\mathcal{Y}=\Lambda$ and $\mathcal{X}=\mathcal{T}$, so $\mathcal{Z}=\mathcal{G}$ is a semiabelian scheme (cf.  \cite{Del}, \cite[Def. 4.1., p. 675]{Jan}).
\end{remark}
\begin{definition}
A family of Deligne 1-motives $\{[\Lambda_\alpha\rightarrow \mathcal{G}]\}_\alpha$ over $\spec\mathcal{O}_{F,S}$ will be called \textbf{a family of wild 1-motives} if 1-motives in this family are pairwise non-isomorphic but after base change to $\spec k_v$ for $v\notin S$, torsion 1-motives $\{[r_v(\Lambda_\alpha)\rightarrow \mathcal{G}_v]\}_\alpha$ become all equal, where $r_v$ is the reduction map for $v\notin S$.
\end{definition}

\subsection{Former constructions of wild 1-motives}\label{subsec1.2}
Let us recall that in 1975, A. Schinzel \cite[pp. 419-420]{Sch} considered four linearly independent points:
\[
\gamma := \begin{bmatrix}
1\\
4
\end{bmatrix}, \,\, \gamma_1 := \begin{bmatrix}
2\\
1
\end{bmatrix}, \,\, \gamma_2 := \begin{bmatrix}
3\\
2
\end{bmatrix}, \,\, \gamma_3 := \begin{bmatrix}
1\\
3
\end{bmatrix}
\]
in $\mathbb{Z}[\frac{1}{2},\frac{1}{3}]^{\times}\times \mathbb{Z}[\frac{1}{2},\frac{1}{3}]^{\times}$. This leads to two lattices $\Gamma_0\subset \Gamma$ as follows: 
\[
\Gamma := \gamma^{\Z} \cdot \gamma_1^{\Z} \cdot \gamma_2^{\Z} \cdot \gamma_3^{\Z},
\]
\[
\Gamma_0 := \gamma_1^{\Z} \cdot \gamma_2^{\Z} \cdot \gamma_3^{\Z}.
\]
A. Schinzel proved that $r_p(\Gamma)=r_p(\Gamma_0)$ for $p\ne 2,3$. He also checked that the equality holds for $p=2,3$.  
Consider $\mathcal{T}=\mathbb{G}_m\times_{{\rm{spec}} \, \mathbb{Z}[\frac{1}{2},\frac{1}{3}]}\mathbb{G}_m$. We have two 1-motives $[\Gamma\rightarrow \mathcal{T}]$ and $[\Gamma_0\rightarrow \mathcal{T}]$ which give two equal torsion 1-motives $[r_p(\Gamma)\rightarrow T_p]=[r_p(\Gamma_0)\rightarrow T_p]$ for each $p>3$. The 1-motive $[\Gamma_0\rightarrow \mathcal{T}]$ was called by P. Jossen \cite{J2} \textbf{the Schinzel's 1-motive}.

In \cite[Remark 3.3, p. 151]{BB}, we generalized the Schinzel's construction and constructed the specific family of four wild 1-motives $[\Lambda\rightarrow \mathcal{T}]$, $[\Lambda'\rightarrow \mathcal{T}]$, $[\Lambda\cap\Lambda'\rightarrow \mathcal{T}]$ and $[\Lambda\cdot\Lambda'\rightarrow \mathcal{T}]$ over $\Z[\frac{1}{2},\frac{1}{3},\frac{1}{5}]$, where $T:= \mathbb{G}^2_{m}$ and $\mathcal{T}:=\mathbb{G}_{m} \times_{{\rm{spec}} \, \mathbb{Z}[\frac{1}{2},\frac{1}{3},\frac{1}{5}]} \mathbb{G}_{m}$. Our extension of Schinzel's example gives lattices $\Lambda,\,\Lambda'$ such that $\Lambda\not\subset \Lambda'$ and $\Lambda'\not\subset \Lambda$.

In the setting of elliptic curves, we constructed families of wild 1-motives in \cite[Remark 3.5, p. 152]{BB}.
Let $E_d=y^2=x^3-d^2x$, $d\in\mathbb{Z}$ be the CM elliptic curve over $\mathbb{Q}$. Take $d$ such that ${\rm rank}_{\mathbb{Z}[i]} E_d(\mathbb{Q}(i))\geq 3$ (cf. \cite[Table 2, p. 464]{RS}, \cite[pp. 330--331]{BK}). Let $Q_1,Q_2,Q_3\in E_d(\mathbb{Q}(i))$ be linearly independent over $\mathbb{Z}[i]$. Let $A=E_d^2$ over $F=\mathbb{Q}(i)$.
Let $S$ be the set of primes of bad reduction of $A$. Let $\mathcal{A}$ be the N{\'e}ron model of $A$ over $\mathcal{O}_{F,S}$.
Consider the following $1$-motives over ${\rm spec}\,\mathcal{O}_{F,S}$: $[\Lambda\rightarrow \mathcal{A}]$, $[\Lambda'\rightarrow \mathcal{A}]$, $[\Lambda\cap\Lambda'\rightarrow \mathcal{A}]$ and $[\Lambda+\Lambda'\rightarrow \mathcal{A}]$, in the sense of P. Deligne, where
\begin{align*}
\Lambda&:=\mathbb{Z}[i]\begin{bmatrix}0\\Q_1\end{bmatrix}+\mathbb{Z}[i]\begin{bmatrix}Q_1\\0\end{bmatrix}+\mathbb{Z}[i]\begin{bmatrix}Q_2\\Q_1\end{bmatrix}+\mathbb{Z}[i]\begin{bmatrix}Q_3\\Q_2\end{bmatrix}+\mathbb{Z}[i]\begin{bmatrix}0\\Q_3\end{bmatrix}, \\
\Lambda'&:=\mathbb{Z}[i]\begin{bmatrix}Q_1\\0\end{bmatrix}+\mathbb{Z}[i]\begin{bmatrix}Q_2\\Q_1\end{bmatrix}+\mathbb{Z}[i]\begin{bmatrix}Q_3\\Q_2\end{bmatrix}+\mathbb{Z}[i]\begin{bmatrix}0\\Q_3\end{bmatrix}+\mathbb{Z}[i]\begin{bmatrix}Q_3\\0\end{bmatrix}
\end{align*}
Changing the base to ${\rm spec}\,k_v$ and taking the images via the reduction maps $r_v$ (for $v\notin S$) of the subgroups $\Lambda,\;\Lambda',\;\Lambda\cap\Lambda',\;\Lambda+\Lambda'$ in $A_v(k_v)$, we obtain torsion 1-motives $[r_v(\Lambda)\rightarrow A_v]$, $[r_v(\Lambda')\rightarrow A_v]$, $[r_v(\Lambda\cap\Lambda')\rightarrow A_v]$ and $[r_v(\Lambda+\Lambda')\rightarrow A_v]$. For each $v\notin S$ these four torsion 1-motives are all equal. 

Another construction of wild 1-motives in the setting of elliptic curves can be built based on \cite{JP}. Let $E$ be an elliptic curve defined over a number field $F$ without CM such that there exist points $P_1,P_2,P_3\in E(F)$ which are linearly independent over $\Z$. Let $A:=E^3$. We define a subgroup $\Lambda\subset A(F)$ and a point $P\in A(F)$ as in \cite{JP}. P. Jossen and A. Perucca have proven that for any $v$ of good reduction $r_v(P)\in r_v(\Lambda)$ but $P\notin \Lambda$. Moreover, let $\Lambda':=\Z P+\Lambda$. Then, we see that $\Lambda\cap \Lambda'=\Lambda$ and $\Lambda+\Lambda'=\Lambda'$. So, we can obtain two different 1-motives in the sense of Deligne $[\Lambda\rightarrow \mathcal{A}]$ and $[\Lambda'\rightarrow \mathcal{A}]$, which will give two equal torsion 1-motives. 
 \begin{remark}
The construction of the wild 1-motive for semiabelian variety which is isogenous to a product of 2-dimensional torus and abelian variety, the second author has presented in \cite[Remark 6.4, p. 248]{B}. This construction links previous constructions in this section.
\end{remark}


 
\section{New constructions of wild 1-motives} \label{new constructions}
Let $n\in\N$, $n\geq 3$.
Let $p_1,p_2,\ldots,p_n$ be different prime numbers. Let $F = \Q$ and $S = \{p_1, p_2,  \ldots, p_n\}$ then $\Z_S= \Z[\frac{1}{p_1},\frac{1}{p_2},\ldots, \frac{1}{p_n}]$.
Consider the group scheme 
$\mathbb{G}_{m}^{\;2} := \mathbb{G}_{m} \times_{{\rm{spec}} \, \Z_S} \mathbb{G}_{m}.$ 
Let $N, N'$ be arbitrary natural numbers. 
Consider the following elements in $\mathbb{G}_m^{\;2}(\Z_S)=\Z_S^\times\times \Z_S^\times$:
\begin{align}
\begin{split}
\lambda &= \begin{bmatrix}
1\\
p_1^{N}
\end{bmatrix}, \; \lambda_1 = \begin{bmatrix}
p_1\\
1
\end{bmatrix}, \; \lambda_2 = \begin{bmatrix}
p_2\\
p_1
\end{bmatrix},\; \lambda_3 = \begin{bmatrix}
p_3\\
p_2
\end{bmatrix},\;\ldots,
\\
\lambda_{n-1} &= \begin{bmatrix}
p_{n-1}\\
p_{n-2}
\end{bmatrix},\;
\lambda_n = \begin{bmatrix}
p_n\\
p_{n-1}
\end{bmatrix}, \; \lambda_{n+1} = \begin{bmatrix}
1\\
p_n
\end{bmatrix}, \; \lambda' = \begin{bmatrix}
p_n^{N'}\\
1
\end{bmatrix}.\end{split}
\label{points in Gm^2}\end{align}
Consider the following subgroups of $\Z_S^\times\times \Z_S^\times :$
\begin{align}
\begin{split}
\Lambda :=& \;\lambda^{\Z} \cdot \lambda_1^{\Z} \cdot \lambda_2^{\Z} \cdot \lambda_3^{\Z}\cdot\ldots\cdot\lambda_{n}^{\Z}\cdot \lambda_{n+1}^{\Z},\\
\Lambda' :=&\; \lambda_1^{\Z} \cdot \lambda_2^{\Z} \cdot \lambda_3^{\Z}\cdot\ldots\cdot\lambda_{n}^{\Z}\cdot \lambda_{n+1}^{\Z}\cdot \lambda'^{\Z}.
\end{split}
\label{subgroups in Gm^2}
\end{align}
\begin{remark}
Observe that $\lambda, \lambda_1,\lambda_2,\ldots,\lambda_{n+1},\lambda'$ are linearly independent over $\Z$ in $\Z_S^\times\times \Z_S^\times$, so $\lambda,\lambda'\notin \Lambda\cap\Lambda'$ and:
\begin{align*}
\Lambda\cap \Lambda'&= \lambda_1^{\Z} \cdot \lambda_2^{\Z} \cdot \lambda_3^{\Z}\cdot\ldots\cdot\lambda_{n}^{\Z}\cdot \lambda_{n+1}^{\Z},\\
\Lambda\cdot \Lambda'&= \lambda^{\Z}\cdot\lambda_1^{\Z} \cdot \lambda_2^{\Z} \cdot \lambda_3^{\Z}\cdot\ldots\cdot\lambda_{n}^{\Z}\cdot \lambda_{n+1}^{\Z}\cdot\lambda'^{\Z}.
\end{align*}
Hence 
\begin{equation}
{\rm rk}_{\Z}\, \Lambda = {\rm rk}_{\Z}\, \Lambda' = n+2,
\label{rank}\end{equation}
\begin{equation}
{\rm rk}_{\Z}\, \Lambda\cap\Lambda'= n+1 \quad \text{and} \quad {\rm rk}_{\Z}\, \Lambda\cdot\Lambda'= n+3.
\label{rank cap}
\end{equation}
\label{ranks of lambdas}
\end{remark}
\begin{theorem} \label{wild 1-motives for tori}
Let $\mathcal{T}=\mathbb{G}_m^2=\mathbb{G}_m\times_{\spec \,\mathbb{Z}_S}\mathbb{G}_m$.
Consider the following 1-motives over $\rm{spec}\, \mathbb{Z}_S$: $[\Lambda\rightarrow \mathcal{T}]$, $[\Lambda'\rightarrow \mathcal{T}]$, $[\Lambda\cap\Lambda'\rightarrow \mathcal{T}]$ and $[\Lambda\cdot\Lambda'\rightarrow \mathcal{T}]$ in the sense of P. Deligne. Then after changing the base to ${\rm spec}\,\mathbb{F}_p$ (for $p\notin S$) and taking the images of the subgroups $\Lambda,\;\Lambda',\;\Lambda\cap\Lambda',\;\Lambda\cdot\Lambda'$ in $T_{p}(\mathbb{F}_p)$ via $r_p$, the torsion 1-motives  $[r_p(\Lambda)\rightarrow T_p]$, $[r_p(\Lambda')\rightarrow T_p]$, $[r_p(\Lambda\cap\Lambda')\rightarrow T_p]$ and $[r_p(\Lambda\cdot\Lambda')\rightarrow T_p]$ are all equal, for each $p\notin S$.
\end{theorem}
The proof of Theorem \ref{wild 1-motives for tori} is a consequence of the following technical result.

\begin{theorem}
For each $p \notin S$, let 
\[
r_p:\; \Z_S^\times\times \Z_S^\times\rightarrow \F_p^\times\times \F_p^\times
\] 
be the natural reduction map. Then for all $p\notin S$:
\[
r_p(\lambda)\in r_p(\Lambda\cap \Lambda') \quad \text{and}\quad r_p(\lambda ')\in r_p(\Lambda\cap \Lambda').
\]
\label{thm}
\end{theorem}
\begin{proof}
Let $p\notin S$.
Let $\overline{\lambda}:=r_p(\lambda), \,\overline{\lambda'}:=r_p(\lambda')$ and $\overline{\lambda_i}:=r_p(\lambda_i)$ for $i=1,\ldots, n+1$.
The equation
\begin{equation}
\overline{\lambda}=\overline{\lambda_1}^{\,R_1}\,\overline{\lambda_2}^{\,R_2}\,\overline{\lambda_3}^{\,R_3}\,\ldots\,\overline{\lambda_{n+1}}^{\,R_{n+1}}
\label{rem 1 for 2 primes}\end{equation}
is equivalent to the system of congruences:
\begin{align}
\label{ncongruencesmodp}
\begin{split}
1&\equiv p_1^{R_1}\,p_2^{R_2}\,\ldots\,p_n^{R_n}\mod{p},\\
p_1^N&\equiv p_1^{R_2}\,p_2^{R_3}\,\ldots\,p_n^{R_{n+1}}\mod{p}.
\end{split}
\end{align}
Let $g$ be a generator of $(\Z/p)^\times$, i.e.:
\[
(\Z/p)^\times \, = \,\, <g> \,\, \cong\Z/(p-1).
\]
The system of congruences \eqref{ncongruencesmodp} is equivalent to the system:
\begin{align}
\label{ncongruences}
\begin{split}
0&\equiv c_1R_1+c_2R_2+\ldots+c_nR_n\mod{(p-1)},\\
Nc_1&\equiv c_1R_2 + c_2R_3+\ldots+c_nR_{n+1}\mod{(p-1)}, 
\end{split}
\end{align}
where $p_1=g^{c_1}, \, p_2=g^{c_2},\,\ldots,\,p_n=g^{c_n}$ for some $c_1,\,c_2,\,\ldots, \,c_n$.
The proof proceeds in two steps.

\textbf{Step 1.}
Assume that $c_i\not\equiv 0\pmod{p-1}$ for all $i=1,\ldots,n$.
Let $D_0:=\gcd(c_1,c_2,\ldots,c_n)$. Then $c_i=k_iD_0$ for $i=1,2,\ldots,n$, with $\gcd(k_1,k_2,\ldots,k_n)=1$.
Define 
\begin{equation}
D:=\gcd\left(\frac{c_1^2}{D_0}, \frac{c_1c_2}{D_0},\ldots,\frac{c_1c_{n-1}}{D_0},c_n\right).
\label{defD}\end{equation}
Observe that:
\[
D=D_0\cdot\gcd(k_1^2,k_1k_2,k_1k_3,\ldots,k_1k_{n-1},k_n).
\]
Let $d:=\gcd(k_1^2,k_1k_2,k_1k_3,\ldots,k_1k_{n-1},k_n)$.
First we show that: $d\mid k_1$.
Let $q$ be a prime number such that $q|d$.
Then $q\mid k_1^2,\; q\mid k_1k_2,\; q\mid k_1k_3,\;\ldots,\; q\mid k_1k_{n-1},\; q\mid k_n$.
By $q \mid k_1^2$, we obtain $q\mid k_1$.
Assume that $q^m\mid d$ for some integer $m>1$.
Then as above $q^m\mid k_1^2$.
Suppose that $q^m\nmid k_1$. 
Then $q^r\| k_1$ (i.e., $q^r\mid k_1$ and $q^{r+1}\nmid k_1$) for an integer $1\leq r< m$.  
Thus $q^{m-r}\mid k_2,\; q^{m-r}\mid k_3,\; \ldots,\; q^{m-r}\mid k_{n-1},\; q^{m}\mid k_n$.
So $q\mid\gcd(k_1,k_2,\ldots,k_n)$.
Contradiction. So $q^m\mid k_1$. It implies that $d\mid k_1$.
Hence:
\begin{equation}
D\mid c_1.
\label{Ddividesc_1}\end{equation}

By \eqref{defD} and \eqref{Ddividesc_1}, the equation:
\[
\frac{c_1^2}{D_0}x_2+\frac{c_1c_2}{D_0}x_3+\ldots+\frac{c_1c_{n-1}}{D_0}x_n+c_nR_{n+1}=Nc_1
\]
has integer solution:
\[
(x_2,x_3, \ldots, x_n, R_{n+1})=(r_2,r_3,\ldots,r_n, r_{n+1}).
\]
Putting
\begin{align*}
R_1&=-\frac{c_2}{D_0}r_2-\frac{c_3}{D_0}r_3-\ldots-\frac{c_n}{D_0}r_n,\\ 
R_2&=\frac{c_1}{D_0}r_2,\\ 
R_3&=\frac{c_1}{D_0}r_3,\\ 
&\vdots\\
R_n&=\frac{c_1}{D_0}r_n,\\ 
R_{n+1}&=r_{n+1}
\end{align*}
we observe, by simple computation, that $R_1 , R_2 ,\ldots, R_{n+1}$ satisfy the system of congruences \eqref{ncongruences}.

\textbf{Step 2.}
Let $j\in\{1,2,\ldots,n\}$ be the smallest number such that $c_j \equiv 0\pmod{p-1}$.
\begin{itemize}
\item If $j=1$ then we have the following solution:
\[
(R_1,R_2,R_3,\ldots, R_n, R_{n+1})=(0,0,0,\ldots,0,0).
\]
\item If $j=2$ then we have the following solution:
\[
(R_1,R_2,R_3,R_4,\ldots, R_n, R_{n+1})=(0,N,0,0,\ldots,0,0).
\]
\item If $j=3$ then put $R_i=0$ for all $4\leq i\leq n+1$. To compute $R_1, R_2, R_3$ follow A.~Schinzel construction \cite[pp. 419-420]{Sch} or the construction from step 1. For convenience of the reader we give computations below. Now, \eqref{ncongruences} has the following form:
\begin{align*}
0&\equiv c_1R_1+c_2R_2\mod{(p-1)},\\
Nc_1&\equiv c_1R_2 + c_2R_3\mod{(p-1)}, \quad \text{ with } c_1, c_2\ne 0.
\end{align*}
In this case $D_0:=\gcd(c_1, c_2)$, $D:=\gcd(\frac{c_1^2}{D_0}, c_2)$ and $d=(k_1^2, k_2)$, where $c_1=k_1D_0$, $c_2=k_2D_0$ with $\gcd(k_1,k_2)=1$. 
We see that $d=1$ and $D=D_0$, hence $D\mid c_1$. So 
the equation:
\[
\frac{c_1^2}{D_0}x_2+c_2R_{3}=Nc_1
\]
has integer solution $(x_2, R_3)=(r_2, r_3)$. Taking:
\[
R_1=-\frac{c_2}{D_0}r_2,\quad R_2=\frac{c_1}{D_0}r_2, \quad R_3=r_3
\]
we obtain a solution $(R_1, R_2, R_3, R_4, R_5  \ldots, R_{n+1}) = (-\frac{c_2}{D_0}r_2, \frac{c_1}{D_0}r_2, r_3, 0, 0,\ldots, 0)$ of \eqref{ncongruences}.
\item If $j\in\{4,\ldots, n\}$ then set $R_i=0$ for all $j+1\leq i\leq n+1$. Then we obtain the following congruences instead of \eqref{ncongruences}:
\begin{align}
\label{jcongruences}
\begin{split}
0&\equiv c_1R_1+c_2R_2+\ldots +c_{j-1}R_{j-1} \pmod{p-1}\\
Nc_1&\equiv c_1R_2+c_2R_3+\ldots +c_{j-1}R_{j} \pmod{p-1}.
\end{split}
\end{align}
To obtain solution of \eqref{jcongruences} replace $n$ by $j-1$ in \eqref{ncongruences} and apply step 1 with this replacement.
\end{itemize}
\bigskip

Now we must show that $\overline{\lambda'}\in r_p(\Lambda\cap\Lambda')$.
The equation
\begin{equation}
\overline{\lambda'}=\overline{\lambda_1}^{R_{n+1}}\,\overline{\lambda_2}^{\,R_n}\,\overline{\lambda_3}^{\,R_{n-1}}\,\ldots \overline{\lambda_{n+1}}^{\,R_1}
\label{rem 2 for 2 primes}\end{equation}
is equivalent to the system of congruences:
\begin{align}
\label{ncongruences2'}
\begin{split}
1&\equiv p_n^{R_1}\,p_{n-1}^{R_2}\,\ldots\,p_1^{R_{n}}\mod{p},\\
p_n^{N'}&\equiv p_n^{R_2}\,p_{n-1}^{R_3}\,\ldots\,p_1^{R_{n+1}}\mod{p}.
\end{split}
\end{align}
It is equivalent to the system:
\begin{align}
\label{ncongruences2}
\begin{split}
0&\equiv c_1R_1+c_2R_2+\ldots+c_nR_n\mod{(p-1)},\\
N'c_1&\equiv c_1R_2 + c_2R_3+\ldots+c_nR_{n+1}\mod{(p-1)}, 
\end{split}
\end{align}
where $p_n=g^{c_1}, \, p_{n-1}=g^{c_2},\,\ldots, \,p_1=g^{c_{n}}$ in $(\Z/p)^\times \, = \,\, <g> \,\, \cong\Z/(p-1)$.
Observe that the systems of congruences \eqref{ncongruences} and \eqref{ncongruences2} are basically the same (we have $N'$ in place of $N$). Observe that computations concerning solutions of \eqref{ncongruences} do not depend on $N$. Hence using the same computations we obtain a solution $R_1,\,R_2,\,R_3,\,\ldots,\,R_{n+1}$ of \eqref{ncongruences2}.
\end{proof}
\begin{corollary}\label{maincor} For each prime number $p\notin S$ we have the following equality:
\[
r_p(\Lambda)=r_p(\Lambda')=r_p(\Lambda\cap\Lambda')=r_p(\Lambda\cdot \Lambda').
\]
\end{corollary}
\begin{proof}
It follows immediately from Theorem \ref{thm}.
\end{proof}
\begin{remark} The proof of \eqref{rem 1 for 2 primes} and \eqref{rem 2 for 2 primes} is also correct for $n=2$ (cf. the case $j=3$ in the proof or \cite[pp. 419-420]{Sch}). 
In this case:
\[
\lambda = \begin{bmatrix}
1\\
p_1^{N}
\end{bmatrix}, \; 
\lambda_1 = \begin{bmatrix}
p_1\\
1
\end{bmatrix}, \; 
\lambda_2 = \begin{bmatrix}
p_2\\
p_1
\end{bmatrix},\; 
\lambda_3 = \begin{bmatrix}
1\\
p_2
\end{bmatrix},\;
\lambda' = \begin{bmatrix}
p_2^{N'}\\
1
\end{bmatrix},
\]
\begin{align*}
\Lambda :=& \;\lambda^{\Z} \cdot \lambda_1^{\Z} \cdot \lambda_2^{\Z} \cdot \lambda_3^{\Z},\\
\Lambda' :=&\; \lambda_1^{\Z} \cdot \lambda_2^{\Z} \cdot \lambda_3^{\Z}\cdot\lambda'^{\Z}.
\end{align*}
Elements $\lambda,\lambda_1,\lambda_2,\lambda_3$ are independent over $\Z$ and $\lambda_1,\lambda_2,\lambda_3, \lambda'$ are also independent over $\Z$.  Observe that
\[
{\rm rk}_{\Z}\,\Z[{\textstyle\frac{1}{p_1},\frac{1}{p_2}}]^\times\times \Z[{\textstyle\frac{1}{p_1},\frac{1}{p_2}}]^\times =4.
\] 
However there is a significant difference between the case $n=2$ and the case $n>2$.
Namely, for $n=2$, taking for example $N=N'$, we easily compute that:
\[
\Lambda=\Lambda'=\Lambda\cap\Lambda'=\Lambda\cdot\Lambda'.
\]
The reason for the difference between the case $n>2$ and the case $n=2$ is the following. Namely, for all $n\geq 2$:
\[
{\rm rk}_{\Z}\,\Z[{\textstyle\frac{1}{p_1},\frac{1}{p_2},\ldots,\frac{1}{p_n}}]^\times\times \Z[{\textstyle\frac{1}{p_1},\frac{1}{p_2},\ldots, \frac{1}{p_n}}]^\times =2n.
\]
However: 
\begin{align*}
{\rm rk}_{\Z}\, \Lambda= {\rm rk}_{\Z}\, \Lambda'< 2n, \quad \text{for }n>2,\\
{\rm rk}_{\Z}\, \Lambda= {\rm rk}_{\Z}\, \Lambda'= 2n, \quad \text{for }n=2.
\end{align*} 
\end{remark}
Let $t\in \N$. Let $S_1, \ldots, S_t$ be a sequence of pairwise disjoint sets of prime numbers, i.e.:
\[
S_j:=\{p_{1j},\ldots, p_{n_{j}j}\},
\] 
such that $n_j>2$ for each $j$.
For fixed $j$, let $N_j,N_j'\in\N$. In the same way as in \eqref{points in Gm^2}, consider the following elements in $\mathbb{G}_m^{\;2}(\Z_{S_j})=\Z_{S_j}^\times\times \Z_{S_j}^\times$:
\[
\lambda_j = \begin{bmatrix}
1\\
p_{1j}^{N_j}
\end{bmatrix}, \; \lambda_{1j} = \begin{bmatrix}
p_{1j}\\
1
\end{bmatrix}, \; \lambda_{2j} = \begin{bmatrix}
p_{2j}\\
p_{1j}
\end{bmatrix},\; \lambda_{3j} = \begin{bmatrix}
p_{3j}\\
p_{2j}
\end{bmatrix},\;\ldots,
\]
\[
\lambda_{(n_j-1)j} = \begin{bmatrix}
p_{(n_j-1)j}\\
p_{(n_j-2)j}
\end{bmatrix},\;
\lambda_{n_jj} = \begin{bmatrix}
p_{n_jj}\\
p_{(n_j-1)j}
\end{bmatrix}, \; \lambda_{(n_j+1)j} = \begin{bmatrix}
1\\
p_{n_jj}
\end{bmatrix}, \; \lambda_j' = \begin{bmatrix}
p_{n_jj}^{N_j'}\\
1
\end{bmatrix}.
\]
Consider the following subgroups of $\Z_{S_j}^\times\times \Z_{S_j}^\times :$
\begin{align}
\begin{split}
\Lambda_j :=& \;\lambda_j^{\Z} \cdot \lambda_{1j}^{\Z} \cdot \lambda_{2j}^{\Z} \cdot \lambda_{3j}^{\Z}\cdot\ldots\cdot\lambda_{n_jj}^{\Z}\cdot \lambda_{(n_j+1)j}^{\Z},\\
\Lambda_j' :=&\; \lambda_{1j}^{\Z} \cdot \lambda_{2j}^{\Z} \cdot \lambda_{3j}^{\Z}\cdot\ldots\cdot\lambda_{n_jj}^{\Z}\cdot \lambda_{(n_j+1)j}^{\Z}\cdot \lambda_j'^{\Z}.
\end{split}\label{groups lambda j and j'}\end{align}
Let 
\begin{equation}
\Sigma_t :=\bigsqcup_{j=1}^t S_j.
\label{sigma_t}\end{equation}
Let 
\begin{align}
\begin{split}
\Gamma_t&:=\Lambda_1\cdot \ldots\cdot\Lambda_t,\\
\Gamma_t'&:=\Lambda_1'\cdot \ldots\cdot\Lambda_t'
\label{gamma_t}\end{split}
\end{align}
be subgroups of $\mathbb{G}_m^{\;2}(\Z_{\Sigma_t})=\Z_{\Sigma_t}^\times\times \Z_{\Sigma_t}^\times$.
Observe that:
\[
\Gamma_t\cap \Gamma_t'=(\Lambda_1\cap\Lambda_1')\cdot \ldots\cdot(\Lambda_t\cap \Lambda'_t),
\]
\[
\Gamma_t\cdot \Gamma_t'=\Lambda_1\cdot\Lambda_1'\cdot \ldots\cdot\Lambda_t\cdot \Lambda'_t.
\]
Notice also that:
\[
{\rm rk}_{\Z}\,\Gamma_t = {\rm rk}_{\Z}\,\Gamma_t' =\sum_{j=1}^t\,(n_j+2),
\]
\[
{\rm rk}_{\Z}\,\Gamma_t\cap \Gamma_t'=\sum_{j=1}^{t}\,(n_j+1),
\]
\[
{\rm rk}_{\Z}\,\Gamma_t\cdot \Gamma_t' =\sum_{j=1}^t \,(n_j+3)
\]
and
\begin{equation}\label{formula for rank of tail}
{\rm rk}_{\Z}\,\Gamma_t/(\Gamma_t\cap \Gamma_t') = {\rm rk}_{\Z}\,\Gamma_t'/(\Gamma_t\cap \Gamma_t')  =t.   
\end{equation}
\begin{corollary}
\label{cor3.1}
For each prime number $p\notin\Sigma_t$, we have the following equality:
\[
r_p(\Gamma_t)=r_p(\Gamma_t')=r_p(\Gamma_t\cap\Gamma_t')=r_p(\Gamma_t\cdot\Gamma_t').
\]
\end{corollary}
\begin{proof}
It follows by Corollary \ref{maincor} (cf. the proof of Theorem \ref{thm}).
\end{proof}
\begin{remark}\label{rem about tail} Let $T_1, T_2\subset\{1,2,\ldots,t\}$ be such that $T_1\cap T_2=\emptyset$. Observe that we can construct groups $\Lambda_j, \Lambda'_j, \Lambda_i, \Lambda'_i$ (cf. \eqref{groups lambda j and j'}) for $j\in T_1$ and $i\in T_2$ in a very similar way to \cite[pp.~419-420]{Sch}, i.e., 
\begin{align*}
\Lambda_j&=\lambda_j^{\Z} \cdot \lambda_{1j}^{\Z} \cdot \lambda_{2j}^{\Z} \cdot \lambda_{3j}^{\Z}\cdot\ldots\cdot\lambda_{n_jj}^{\Z}\cdot \lambda_{(n_j+1)j}^{\Z},\\
\Lambda_j'&=\lambda_{1j}^{\Z} \cdot \lambda_{2j}^{\Z} \cdot \lambda_{3j}^{\Z}\cdot\ldots\cdot\lambda_{n_jj}^{\Z}\cdot \lambda_{(n_j+1)j}^{\Z}
\end{align*}
and 
\begin{align*}
\Lambda_i&=\lambda_{1i}^{\Z} \cdot \lambda_{2i}^{\Z} \cdot \lambda_{3i}^{\Z}\cdot\ldots\cdot\lambda_{n_ii}^{\Z}\cdot \lambda_{(n_i+1)i}^{\Z},\\
\Lambda_i'&=\lambda_{1i}^{\Z} \cdot \lambda_{2i}^{\Z} \cdot \lambda_{3i}^{\Z}\cdot\ldots\cdot\lambda_{n_ii}^{\Z}\cdot \lambda_{(n_i+1)i}^{\Z}\cdot\lambda_i'^{\Z}.
\end{align*}
For $k\in \{1,2,\ldots,t\}\setminus (T_1\cup T_2)$ we define groups $\Lambda_k$ and $\Lambda_k'$ as in \eqref{groups lambda j and j'}. Consider $\Gamma_t$ and $\Gamma_t'$ defined as in \eqref{gamma_t}. In this case, we obtain
\begin{align}
\begin{split}\label{po modyfikacji}
{\rm rk}_{\Z}\,\Gamma_t - {\rm rk}_{\Z}\,\Gamma_t\cap \Gamma_t'&=t-|T_2|,\\
{\rm rk}_{\Z}\,\Gamma_t' - {\rm rk}_{\Z}\,\Gamma_t\cap \Gamma_t'&=t-|T_1|.
\end{split}\end{align}
\end{remark}

 Let $\Sigma_t$ be defined as in \eqref{sigma_t} and let $\Gamma_t$, $\Gamma_t'$ be defined as in \eqref{gamma_t} (cf. Remark \ref{rem about tail}).
We can generalize Theorem \ref{wild 1-motives for tori}.
Namely, in the following Theorem~\ref{const2}, applying \eqref{formula for rank of tail} or \eqref{po modyfikacji}, we will obtain a family of wild 1-motives $[\Gamma_t\rightarrow \mathcal{T}]$, $[\Gamma_t'\rightarrow \mathcal{T}]$, $[\Gamma_t\cap\Gamma_t'\rightarrow \mathcal{T}]$ and $[\Gamma_t\cdot\Gamma_t'\rightarrow \mathcal{T}]$ in the sense of P. Deligne, where ${\rm rk}_{\Z}\,\Gamma_t - {\rm rk}_{\Z}\,\Gamma_t\cap \Gamma_t'$ and ${\rm rk}_{\Z}\,\Gamma_t' - {\rm rk}_{\Z}\,\Gamma_t\cap \Gamma_t'$ can be arbitrary large but after changing the base to ${\rm spec}\,\mathbb{F}_p$ and after reduction for all $p\notin\Sigma_t$ we obtain the same torsion 1-motives.

\begin{theorem}\label{const2}
Let $\mathcal{T}=\mathbb{G}_m^2=\mathbb{G}_m\times_{{\rm spec }\,\mathbb{Z}_{\Sigma_t}}\mathbb{G}_m$. Consider the following 1-motives over $\rm{spec}\, \mathbb{Z}_{\Sigma_t}$: $[\Gamma_t\rightarrow \mathcal{T}]$, $[\Gamma_t'\rightarrow \mathcal{T}]$, $[\Gamma_t\cap\Gamma_t'\rightarrow \mathcal{T}]$ and $[\Gamma_t\cdot\Gamma_t'\rightarrow \mathcal{T}]$ in the sense of P.~Deligne. Then after changing the base to ${\spec}\,\mathbb{F}_p$ (for $p\notin \Sigma_t$) and taking the images of the subgroups $\Gamma_t,\;\Gamma_t',\;\Gamma_t\cap\Gamma_t',\;\Gamma_t\cdot\Gamma_t'$ in $T_{p}(\mathbb{F}_p)$ via $r_p$, the torsion 1-motives  $[r_p(\Gamma_t)\rightarrow T_p]$, $[r_p(\Gamma_t')\rightarrow T_p]$, $[r_p(\Gamma_t\cap\Gamma_t')\rightarrow T_p]$ and $[r_p(\Gamma_t\cdot\Gamma_t')\rightarrow T_p]$ are all equal, for each $p\notin \Sigma_t$.
\end{theorem}
\begin{proof}
It follows by Corollary \ref{cor3.1}.
\end{proof}
\begin{remark}
Observe that we can consider our constructions also for abelian varieties like in subsection \ref{subsec1.2}. But there is a bound for the rank of the Mordell-Weil group of $E_d$.
\end{remark}
\begin{remark} Observe that every semiabelian variety $G$ over $F$, with the toral dimension $\dim T=r\geq 2$:
\[
0\rightarrow T\xrightarrow{i} G\xrightarrow{\pi} A\rightarrow 0,
\]
potentially admits families of wild 1-motives. Indeed, let $L/F$ be such that $T\otimes_F L=\mathbb{G}_{m}^r$. Consider a model:
\[
0\rightarrow \mathbb{G}_m^r\xrightarrow{i} \mathcal{G}\xrightarrow{\pi} \mathcal{A}\rightarrow 0
\]
of $G$ over $\mathcal{O}_{L,S}$ for a set $S$ of prime ideals in $\mathcal{O}_L$ such that $\mathcal{G}$ has good reduction for every $v\notin S$. Then we construct families of wild 1-motives $\{[\Lambda_\alpha\rightarrow \mathcal{G}]\}_\alpha$ applying Theorems \ref{wild 1-motives for tori} or \ref{const2} and the map $i:\mathbb{G}_m^r\rightarrow \mathcal{G}$. 
\end{remark}
\section{Appendix on multiple base discrete logarithm problem}\label{appendix}
One of the most important problems in cryptography is the discrete logarithm problem (see \cite[Chapters 5, 6]{CP}, \cite[Chapter 7]{SP}) concerning solutions of the 
equation $a^x = b$ with $a, b \in \mathbb{F}_p^{\,\times}.$ Naturally discrete logarithm problem can be extended
to any abelian group $G.$ It concerns solutions of $g_{1}^{x} = g$ in $x \in \Z$, where $g_{1} ,g \in G$.
We can also consider natural generalization of discrete logarithm problem to multiple base $(\alpha_1,\ldots,\alpha_n)$. A multiple base $(g_1,\ldots,g_n)$ discrete logarithm problem in an abelian group $G$ concerns solutions of $g_1^{x_1}\cdot\ldots\cdot g_n^{x_n}=g$ in $(x_1,\ldots,x_n)\in\Z^n$, where $g_1,\ldots,g_n,g\in G$. 

In relation to the discrete logarithm problem in $\mathbb{F}_p^{\,\times}$ there is an important local to global 
discrete logarithm problem. It asks whether equation $\alpha^{x} = \beta$ in $\mathcal{O}_{F,S}^{\times}$ has solution in $x \in \Z$  if and only if equation $r_{v}(\alpha)^{x_{v}} = r_{v} (\beta)$ in $k_{v}^{\times}$ has solution in $x_{v} \in \Z$ for almost all prime ideals $v \in \mathcal{O}_{F, S},$ where
$r_{v} : \mathcal{O}_{F,S}^{\times} \rightarrow k_{v}^{\times}$ is the reduction map at $v.$ This problem was solved affirmatively by A.~Schinzel in 1975 \cite[Theorem 2, p. 398]{Sch}. 

One can consider natural generalization of the local to global discrete logarithm problem to multiple base
$(\alpha_1,\ldots,\alpha_n)$. However, here situation is not as simple as for the single base local to global discrete logarithm problem. The constructions of families of wild 1-motives which were discussed in subsection \ref{subsec1.2} give counterexample to multiple base discrete logarithm problem. Namely, for instance, A.~Schinzel \cite[pp. 419-420]{Sch} gave counterexample to multiple base discrete logarithm problem in $\mathbb{G}_m^{\;2}(\Z[\frac{1}{2},\frac{1}{3}]).$  In \cite[pp. 330-332]{BK}, the Schinzel's counterexample \cite[pp. 419-420]{Sch} was extended to abelian surfaces $E_d^{\;2}$, where $E_d$ is the CM elliptic curve given by the equation $y^2=x^3-d^2x$ for $d\in\Z$. In \cite[pp. 149-152]{BB}, the Schinzel's counterexample was generalized to 
$\mathbb{G}_m^{\;2}(\Z[\frac{1}{2},\frac{1}{3}, \frac{1}{5}])$, etc.

Based on section \ref{new constructions}, we show that the local to global multiple base $(\alpha_1,\ldots,\alpha_n)$ discrete logarithm problem in $\mathbb{G}_m^{\;2}(\mathcal{O}_{F,S}) = (\mathcal{O}_{F,S}^\times)^{2}$ has counterexamples for arbitrary big bases, including infinite bases (Theorem \ref{thm} cf. Remark \ref{ranks of lambdas}, Corollaries \ref{cor3.1}, \ref{cor3.3}). For this reason multiple base discrete logarithm problem is much more resistant to breaking encryption algorithms.

%
\subsection{The local to global multiple base $(\alpha_1,\ldots,\alpha_n)$ discrete logarithm problem}
 With notation from section \ref{new constructions}, let us describe consequences of results of that section.
\begin{remark}
For any set of different prime numbers $S=\{p_1, p_2, \ldots, p_n\}$ we can find such natural numbers $N, N'$ that systems of congruences \eqref{ncongruencesmodp}, \eqref{ncongruences2'} will have solutions also for $p\in S$. Indeed, let $N=N':=\varphi(p_1 p_2\ldots p_n)$, where $\varphi(\cdot)$ is the Euler function. Then we obtain the following solutions:
\begin{center}
\begin{tabular}{|c|c|}
\hline
&solutions of \eqref{ncongruencesmodp} modulo $p_i$\\
&$(R_1,\,R_2,\,R_3,\ldots,\,R_n, \,R_{n+1})$\\
\hline
$i=1$& $(0,\,\varphi(p_1),\,0,\ldots,\,0,\,0)$\\
\hline
$i\ne 1$& $(0,\,0,\,0,\ldots,\,0,\,0)$\\
\hline
\end{tabular}\\[0.3cm]
\begin{tabular}{|c|c|}
\hline
&solutions of \eqref{ncongruences2'} modulo $p_i$\\
& $(R_1,\,R_2,\,R_3,\ldots,\,R_n, \,R_{n+1})$\\
\hline
$i=n$& $(0,\,\varphi(p_n),\,0,\ldots,\,0,\,0)$\\
\hline
$i\ne n$& $(0,\,0,\,0,\ldots,\,0,\,0)$\\
\hline
\end{tabular}\\[0.3cm]
\end{center}
\label{rem2.4}\end{remark}
\begin{corollary} Let $N=N'=\varphi(p_1 p_2\ldots p_n)$. 
By Theorem \ref{thm} and Remark \ref{rem2.4}, multiple base $(\lambda_1,\ldots,\lambda_{n+1})$ discrete logarithm of the elements $\lambda$ and $\lambda'$ has solutions modulo $p$, for all prime numbers $p$. However, multiple base $(\lambda_1,\ldots,\lambda_{n+1})$ discrete logarithm of $\lambda$ and $\lambda'$ does not have solutions. 
\end{corollary}
\begin{corollary}
Let $N_j=N_j'=\varphi(p_{1j}p_{2j}\ldots p_{n_{j}j})$ for each $j=1,\ldots, t$. Consider sequence $b_t$ which is the concatenation of the sequences $(\lambda_{1j},\ldots,\lambda_{(n_{j}+1)j})$ for $j=1,\ldots, t$. By Theorem \ref{thm} and Remark \ref{rem2.4}, multiple base $b_t$ discrete logarithm of each $\lambda_j$ and $\lambda_j'$, $j=1,\ldots t$, has solutions modulo $p$, for every $p$. However, multiple base $b_t$ discrete logarithm of each $\lambda_j$ and $\lambda_j'$, $j=1,\ldots, t$, does not have solutions. 
\end{corollary}
\subsubsection{Infinite basis} Let $\mathcal{P}:=\{2,3,5, \ldots, p,\ldots\}$ be the set of all prime numbers.
Now, consider infinite sequence of pairwise disjoint sets of prime numbers:
\begin{equation}\label{definition of infinite sequence}
S_1=\{p_{11},\ldots, p_{n_{1}1}\}, \ldots, S_t=\{p_{1t},\ldots, p_{n_{t}t}\},\ldots.
\end{equation}
Define $\Sigma_t$ as in \eqref{sigma_t} and put 
\[
\Sigma_\infty:=\bigcup_{t=1}^\infty\Sigma_t=\bigsqcup_{t=1}^\infty S_t.
\]
Define subgroups $\Gamma_t, \Gamma_t'$ of $\mathbb{G}_m^{\;2}(\Z_{\Sigma_\infty})=\Z_{\Sigma_\infty}^\times\times \Z_{\Sigma_\infty}^\times$ as in \eqref{gamma_t}. Observe that:
\[
\Gamma_1\subset\ldots\subset \Gamma_t\subset\Gamma_{t+1}\subset\ldots,
\]
\[
\Gamma_1'\subset\ldots\subset \Gamma_t'\subset\Gamma_{t+1}'\subset\ldots.
\]
Let
\begin{align*}
\Gamma_\infty&:=\varinjlim_{t}\Gamma_t=\bigcup_{t=1}^\infty \Gamma_t,\\
\Gamma_\infty'&:=\varinjlim_{t}\Gamma_t'=\bigcup_{t=1}^\infty \Gamma_t'.
\end{align*}
Notice that
\[
\Gamma_\infty\cap\Gamma_\infty'=\varinjlim_t \Gamma_t\cap\Gamma_t'=\bigcup_{t=1}^\infty\;\Gamma_t\cap\Gamma_t'
\]
and
\[
\Gamma_\infty\cdot\Gamma_\infty'=\varinjlim_t \Gamma_t\cdot\Gamma_t'=\bigcup_{t=1}^\infty\;\Gamma_t\cdot\Gamma_t'.
\]
Observe that
\[
{\rm rk}_{\Z}\,\Gamma_\infty = {\rm rk}_{\Z}\,\Gamma_\infty' = {\rm rk}_{\Z}\,\Gamma_\infty\cap \Gamma_\infty'  = {\rm rk}_{\Z}\,\Gamma_\infty\cdot \Gamma_\infty' =\infty   
\]
and
\[
{\rm rk}_{\Z}\,\Gamma_\infty/(\Gamma_\infty\cap \Gamma_\infty') = {\rm rk}_{\Z}\,\Gamma_\infty'/(\Gamma_\infty\cap \Gamma_\infty')  =\infty.   
\]
\begin{corollary}\label{cor3.3}
Assume that $\mathcal{P}\setminus \Sigma_\infty\ne \emptyset$. For each prime number $p\notin \Sigma_\infty$, we have the following equality:
\[
r_p(\Gamma_\infty)=r_p(\Gamma_\infty')=r_p(\Gamma_\infty\cap\Gamma_\infty')=r_p(\Gamma_\infty\cdot\Gamma_\infty').
\]
\end{corollary}
\begin{proof}
It follows by Corollary \ref{cor3.1} and construction of $\Gamma_\infty, \Gamma_\infty'$.
\end{proof}
\begin{remark}
In Corollaries \ref{maincor}, \ref{cor3.1} and \ref{cor3.3} we had to exclude prime numbers in $S, \Sigma_t$ and $\Sigma_\infty$, respectively, because we consider reduction of groups. In this case we also had to consider the reduction of inverses of elements of these groups.
\end{remark}
\begin{corollary}
Let $N_j=N_j'=\varphi(p_{1j}p_{2j}\ldots p_{n_{j}j})$ for each $j\in\N$. Consider sequence $b_\infty$ which is the concatenation of the sequences $(\lambda_{1j},\ldots,\lambda_{(n_j+1)j})$ for $j\in\N$. By Theorem \ref{thm} and Remark \ref{rem2.4}, multiple base $b_\infty$ discrete logarithm of each $\lambda_j$ and $\lambda_j'$, $j\in\N$, has solutions modulo $p$, for every $p$. However, multiple base $b_\infty$ discrete logarithm of each $\lambda_j$ and $\lambda_j'$, $j\in\N$, does not have solutions. 
\label{rem3.4}\end{corollary}
\begin{remark}
Observe that Remark \ref{rem3.4} also works in the case where $\Sigma_\infty=\mathcal{P}$. In this case $\Z_{\Sigma_\infty}=\Q$ but the reader observes that it is not an obstruction to run down our computations.
\end{remark}
%
%
%
%
%
%
%
%
%
%
%
%
%
%
%

\end{document}